\documentclass [11pt,twoside,a4paper]{article}
\usepackage{amsmath}
\usepackage{}
\usepackage{mathrsfs}
\usepackage{amssymb}
\usepackage{amsfonts}
\usepackage{amsmath,amssymb, bm}
\usepackage{enumerate}
\usepackage{color}
\usepackage{bbm}
\usepackage{soul}
\usepackage{url}
\newfont{\aaa}{cmb10 at 19pt}
\newfont{\bbb}{cmb10 at 11pt}

\newcommand{\ind}{{\mathrm i}{\mathrm n}{\mathrm d\,}}

\newcommand{\Ker}{{\mathrm K}{\mathrm e}{\mathrm r\,}}

\newcommand{\Supp}{{\mathrm S}{\mathrm u}{\mathrm p}{\mathrm p\,}}

\newtheorem{thm}{Theorem}[section]
\newtheorem{prop}[thm]{Proposition}

\newtheorem{cor}[thm]{Corollary}

\newenvironment{proof}{\noindent{\it Proof}.\ \ }{\hfill $ \square $\vskip 4mm}

\numberwithin{equation}{section}

\title{A Lichnerowicz vanishing theorem for proper cocompact actions}
\author{Ziran Liu\\ $\ $
\\ {\footnotesize Chern Institute of Mathematics \& LPMC,
Nankai University, Tianjin 300071,
P. R. China}\\
{\footnotesize \url{ liuziran@mail.nankai.edu.cn}}}
\date{}
\begin{document}
\maketitle
\begin{abstract}
We establish a Lichnerowicz type vanishing theorem for non-compact spin manifolds admiting proper cocompact actions, when the action group is unimodular.
\end{abstract}

\section{Introduction}\label{sec:introduction}
A well-known result of Lichnerowicz \cite{Lichnerowicz1963} states that if a compact spin manifold of dimension $4k$ admits a Riemannian metric of positive scalar curvature, then the index of the associated Dirac opertator vanishes.

The purpose of this note is to establish an extension of this classical result to the case of non-compact manifolds admiting proper cocompact group actions. To be more precise, we will prove such a result, when the action group is {\it unimodular}, for the   Mathai-Zhang index introduced in \cite{MathaiZhang2010} for such actions. The main result is stated in Theorem \ref{thm:main}.

This note is organized as follows. In Section~\ref{sec:Z1}, we recall the definition of the Mathai-Zhang index \cite{MathaiZhang2010} for  spin manifolds. In Section~\ref{sec:Z2} we prove our main result mentioned above.

\section{The Mathai-Zhang index on spin manifold}\label{sec:Z1}
Let $M$ be a non-compact even dimensional spin manifold. Let $G$ be a locally compact group.
We assume that $G$ acts on $M$ properly and cocompactly, where by proper we mean that the following map
\begin{equation*}
G \times M \rightarrow M\times M,\qquad (g,x) \mapsto (x,gx),
\end{equation*}
is proper (the pre-image of a compact subset is compact), while by cocompact we mean that the quotient space $M/G$ is compact.

We make the assumption that $G$ preserves the spin structure on $M$.

Let $g^{TM}$ be a Riemannian metric on the tangent vector bundle $TM$. Without loss of generality, we can and we will assume that $g^{TM}$ is $G$-invariant (cf. \cite[(2.3)]{MathaiZhang2010}).

Let $S(TM) = S_+(TM) \oplus S_-(TM)$ be the complex vector  bundle of spinors associated to $(TM, g^{TM})$.
Let $g^{S(TM)}$ and $\nabla^{S(TM)}$ be the canonically induced Hermitian metric and connection on $S(TM)$. Then the action of $G$ on $M$ lifts to an action on $S(TM)$, preserving $g^{S(TM)}$ and $\nabla^{S(TM)}$.

Let $E$ be a Hermitian vector bundle over $M$ such that it admits a $G$-action lifted from the action of $G$ on $M$. Let $g^E$ be a $G$-invariant Hermitian metric on $E$, and $\nabla^E$ be a $G$-invariant Hermitian connection on $E$.

  For any tangent vector $X\in TM$, let the Clifford action $c(X)$ on $S(TM)$    extend to   act on $S(TM)\otimes E$ by acting as identity on $E$, and we still denote this action by $c(X)$. The $G$-invariant metrics and connections on $S(TM)$ and $E$ induce canonically a $G$-invariant Hermitian metric $g^{S(TM)\otimes E}$ and a $G$-invariant connection $\nabla^{S(TM)\otimes E}$ on $S(TM)\otimes E$.

The twisted (by $E$) Dirac operator acting on $\Gamma(S(TM)\otimes E)$ is given by
\begin{equation}\label{eqn:Ds}
D^E=\sum\limits_{i=1}^{\dim M}c(e_i)\nabla^{S(TM)\otimes E}_{e_i}\ :\ \Gamma(S(TM)\otimes E) \rightarrow \Gamma(S(TM)\otimes E),
\end{equation}
where $e_1,\ldots,e_{\dim M}$ is an oriented orthonormal basis of $TM$.

Let $\Gamma(S(TM)\otimes E)$ carry the natural inner product such that for any $s_1, \, s_2\in$ $\Gamma(S(TM)\otimes E)$ with compact support,
\begin{equation}\label{eqn:inner}
(s_1,s_2)=\int_M \langle s_1,s_2\rangle_{S(TM)\otimes E} d{v}_{g^{TM}}.
\end{equation}

Let $\|\cdot\|_0$ be the $L^2$-norm associated to the inner product (\ref{eqn:inner}), let $\|\cdot\|_1$ be a (fixed) $G$-invariant Sobolev $1$-norm. Let $\mathbf{H}^0(M,S(TM)\otimes E)$ be the completion of $\Gamma(S(TM)\otimes E)$ under $\|\cdot\|_0$.

Denote the space of $G$-invariant smooth sections of $S(TM)\otimes E$ by $\Gamma(S(TM)\otimes E)^G$.

Now recall that the compactness of $M/G$ guarantees the existence of a compact subset $Y$ of $M$ such that $G(Y) = M$ (cf. \cite[Lemma 2.3]{Phillips}).
Let $U$, $U'$ be two open subsets of $M$ such that $Y \subset U$ and that the closures $\overline{U}$ and $\overline{U'}$ are both compact in $M$, and that $\overline{U} \subset U'$. The existence of $U$, $U'$ is clear.

Following \cite{MathaiZhang2010}, let  $f\in C^\infty_c(M)$ be a nonnegative function  such that $f|_U = 1$ and $\Supp(f) \subset U'$.

Let $\mathbf{H}^0_f(M,S(TM)\otimes E)^G$ and $\mathbf{H}^1_f(M,S(TM)\otimes E)^G$ be the completions of $\{fs:s\in\Gamma(S(TM)\otimes E)^G\}$ under $\|\cdot\|_0$ and $\|\cdot\|_1$ respectively.

Let $P_f$ denote the orthogonal projection from $\mathbf{H}^0(M,S(TM)\otimes E)$ to its subspace $\mathbf{H}^0_f(M,S(TM)\otimes E)^G$.

Clearly, $P_fD^E$ maps $\mathbf{H}^1_f(M,S(TM)\otimes E)^G$ into $\mathbf{H}^0_f(M,S(TM)\otimes E)^G$.

We now recall a basic result from \cite[Proposition 2.1]{MathaiZhang2010}.
\begin{prop}[Mathai-Zhang] \label{prop:fred}
The operator
\begin{equation*}
P_fD^E:\ \mathbf{H}_f^1(M,S(TM)\otimes E)^G \rightarrow \mathbf{H}_f^0(M,S(TM)\otimes E)^G
\end{equation*}
is a Fredholm operator.
\end{prop}

Let $D_\pm^E:\Gamma(S_\pm(TM)\otimes E) \rightarrow \Gamma(S_\mp(TM)\otimes E)$ be the restrictions of $D^E$ on $\Gamma (S_\pm(TM)\otimes E)$  respectively.
It has been shown in \cite{MathaiZhang2010} that $\ind(P_fD^E_+)$ is independent of the choices of the cut-off function $f$, as well as the $G$-invariant metrics and connections involved. Following \cite[Defition 2.4]{MathaiZhang2010}, we denote $\ind(P_fD^E_+)$ by $\ind_G(D^E_+)$.

\section{A Lichnerowicz vanishing theorem for Mathai-Zhang index}\label{sec:Z2}
In this section, we extend the Lichnerowicz vanishing theorem to the case of $\ind_G(D^E_+)$ when the action group $G$ is {\it unimodular}.


Let   ${k^{TM}}$ be the scalar curvature of $g^{TM}$. Let  ${R}^E$ be the curvature of $\nabla^E$.

Let $\{e_i\}$ be a local orthonormal basis of $TM$. Let $c({R}^E)$, which acts  on $\Gamma(S(TM)\otimes E)$, be defined by
\begin{equation}
c\left({R}^E\right) =\frac{1}{2} \sum_{i,\,j=1}^{\dim M}c(e_i)c(e_j) \, {R}^E(e_i,e_j ).
\end{equation}

Recall the famous   Lichnerowicz formula \cite{Lichnerowicz1963} (cf. \cite[Chapter II, Theorem 8.17]{LawMich}), which states that,
\begin{equation}\label{eqn:Lich}
\left(D^E\right)^2 = -\Delta^E + \frac{{k^{TM}}}{4} + c\left({R}^E\right),
\end{equation}
where $$\Delta^E=\sum_{i=1}^{\dim M}\left(\left(\nabla_{e_i}^{S(TM)\otimes E}\right)^2-\nabla^{S(TM)\otimes E}_{\nabla^{TM}_{e_i}e_i}\right)$$ is the Bochner  Laplacian.

We can now state the  main result   this note   as follows.
\begin{thm}\label{thm:main}
Let $M$ be a non-compact $spin$ manifold carrying a proper cocompact action by a locally compact {\rm unimodular} group $G$, such that the $G$-action preserves the spin structure on $M$.  Let $g^{TM}$ be a $G$-invariant metric on the tangent bundle $TM$. Let $E$ be a Hermitian vector bundle over $M$,  carrying a $G$-action lifted from that on $M$, as well as a $G$-invariant Hermitian metric   and a $G$-invariant Hermitian connection $\nabla^E$.   If  one assumes that $ {k^{TM}} +4\, c({R}^E) > 0$ holds pointwise over $M$,  then the Mathai-Zhang index $\ind_G(D^E_+)$ vanishes.
\end{thm}

\begin{proof}
Since $G$ is {\it unimodular}, by a result of Mathai-Zhang \cite[Theorem 2.7]{MathaiZhang2010}, one has that
\begin{equation}\label{bb}
\ind_G\left(D^E_+\right)=\dim\left(\left(\Ker D_+^E\right)^G\right)-\dim\left(\left(\Ker D_-^E\right)^G\right).
\end{equation}

Set $\mathfrak{R}=\frac{{k^{TM}}}{4} + c({R}^E)$.

For any $s\in (\Ker D^E)^G$ and $\varphi \in C_c^\infty(M)$, by    (\ref{eqn:Lich}) one gets that
\begin{equation}\label{cc}
0 = \int_M \left\langle\left(D^E\right)^2s, \varphi s \right\rangle dv_{g^{TM}}=\int_M \Big( \langle-\Delta^E s, \varphi s\rangle + \varphi \langle\mathfrak{R}s, s\rangle \Big)dv_{g^{TM}}.
\end{equation}

By (\ref{cc}), one has,  where we denote $\nabla^{S(TM)\otimes E}$ by $\nabla$ for simplicity,
\begin{eqnarray}\label{dd}
0 & = &
\int_M\left( \sum\limits_{i=1}^{\dim M} \left\langle\nabla_{e_i}s,\nabla_{e_i}(\varphi s)\right\rangle+ \varphi \langle\mathfrak{R}s, s\rangle \right) dv_{g^{TM}}  \nonumber \\
  & = &
\int_M\left( \sum\limits_{i=1}^{\dim M} e_i(\varphi)\langle\nabla_{e_i}s, s\rangle
  +\sum\limits_{i=1}^{\dim M} \varphi\langle\nabla_{e_i} s,\nabla_{e_i} s\rangle+\varphi\langle\mathfrak{R}s, s\rangle \right) dv_{g^{TM}}\nonumber  \\
  & = &
\int_M\left( \frac{1}{2}\sum\limits_{i=1}^{\dim M} e_i(\varphi)e_i\left(|s|^2\right)
  +\varphi |\nabla s|^2+\varphi\langle\mathfrak{R}s, s\rangle \right) dv_{g^{TM}}\nonumber  \\
  & = &
\int_M\left( \frac{1}{2}\sum\limits_{i=1}^{\dim M} e_i\left(\varphi e_i\left(|s|^2\right)\right)-\frac{1}{2}\sum\limits_{i=1}^{\dim M} \varphi e_i\left(e_i\left(|s|^2\right)\right)
  +\varphi |\nabla s|^2+\varphi\langle\mathfrak{R}s, s\rangle \right) dv_{g^{TM}}\nonumber  \\
  & = &
\int_M \varphi \left( -\frac{1}{2}\Delta\left( |s|^2\right) +|\nabla s|^2+\langle\mathfrak{R}s, s\rangle \right)dv_{g^{TM}}.
\end{eqnarray}

Since $\varphi$ is arbitrary, one sees from (\ref{dd}) that for any $s \in (\Ker D^E)^G$, there holds,
\begin{equation}\label{aa}
\frac{1}{2}\Delta\left(|s|^2\right)=|\nabla s|^2+\langle\mathfrak{R}s, s\rangle\geqslant \langle\mathfrak{R}s, s\rangle .
\end{equation}

Since the $G$-action on $M$ is cocompact, and that $|s|^2$ is  clearly $G$-invariant,    $|s|^2$  attains its maximum  at certain point $p\in M$, on which one has, by the standard maximum principle, that
$$\Delta\left(|s|^2\right)\leqslant 0.$$
Combining with  (\ref{aa}), one sees that under the assumption that
$\mathfrak{R}>0$ over $M$,
\begin{equation}\label{ineq311}
|s(p)|^2=0,
\end{equation}
which implies that $s\equiv 0$ on $M$.

Combining with (\ref{bb}), one completes the proof of Theorem \ref{thm:main}.
\end{proof}

\begin{cor}\label{cor:spin}
Let $M$ be a non-compact $spin$ manifold carrying a proper cocompact action by a locally compact {\rm unimodular} group $G$, such that the $G$-action preserves the spin structure on $M$. If $M$ admits a $G$-invariant Riemannian metric with positive scalar curvature,  then  the Mathai-Zhang index $\ind_G(D_+)$ vanishes.
\end{cor}

$\  $

\noindent\bf{\footnotesize Acknowledgements.}\quad\rm
{\footnotesize The author thanks Professor Weiping Zhang for helpful discussions.}\\[4mm]

\end{document}